\documentclass[dvipsnames,letterpaper,12pt]{article}

\usepackage[margin = 1.0in]{geometry}
\usepackage{amsmath,amssymb,graphicx,mathabx,accents}
\usepackage{enumerate,mdwlist}

\numberwithin{equation}{section}

\usepackage{tikz, tkz-berge, tkz-graph}
\usetikzlibrary{patterns,arrows,decorations.pathreplacing}

\usepackage{color,xcolor}
\definecolor{crimsonred}{RGB}{132,22,23}
\definecolor{darkblue}{RGB}{72,61,139}

\usepackage{amsthm}
\theoremstyle{plain}
\newtheorem{theorem}{Theorem}
\newtheorem{lemma}{Lemma}
\newtheorem{corollary}{Corollary}

\theoremstyle{remark}

\DeclareMathOperator{\hausdim}{\dim_{\mathbf{H}}}
\DeclareMathOperator{\lowminkdim}{\underline{\dim}_{\mathbf{M}}}
\DeclareMathOperator{\upminkdim}{\overline{\dim}_{\mathbf{M}}}

\DeclareMathOperator{\RR}{\mathbf{R}}
\DeclareMathOperator{\ZZ}{\mathbf{Z}}

\DeclareMathOperator{\Prob}{\mathbf{P}}
\DeclareMathOperator{\Expect}{\mathbf{E}}

\DeclareMathOperator{\B}{\mathcal{B}}

\title{Large Sets Avoiding Rough Patterns}
\author{Jacob Denson\thanks{University of British Columbia, Vancouver BC, \{denson, malabika, jzahl\}@math.ubc.ca.} \and Malabika Pramanik\footnotemark[1] \and Joshua Zahl\footnotemark[1]}

\begin{document}

\maketitle

\begin{abstract}
	The pattern avoidance problem seeks to construct a set $X\subset \RR^d$ with large dimension that avoids a prescribed pattern. Examples of such patterns include three-term arithmetic progressions (solutions to $x_1 - 2x_2 + x_3 = 0$), or more general patterns of the form $f(x_1, \dots, x_n) = 0$. Previous work on the subject has considered patterns described by polynomials, or by functions $f$ satisfying certain regularity conditions. We consider the case of `rough' patterns, not necessarily given by the zero-set of a function with prescribed regularity.

	There are several problems that fit into the framework of rough pattern avoidance. As a first application, if $Y \subset \RR^d$ is a set with Minkowski dimension $\alpha$, we construct a set $X$ with Hausdorff dimension $d-\alpha$ such that $X+X$ is disjoint from $Y$. As a second application, if $C$ is a Lipschitz curve, we construct a set $X \subset C$ of dimension $1/2$ that does not contain the vertices of an isosceles triangle.
\end{abstract}

A major question in modern geometric measure theory is whether sufficiently large sets are forced to contain copies of certain patterns. Intuitively, one expects the answer to be yes, and many results in the literature support this intuition. For example, the Lebesgue density theorem implies that a set of positive Lebesgue measure contains an affine copy of every finite set. If $X \subset \RR^d$ has large Hausdorff dimension, then it must contain many points that lie in a lower dimensional plane section (see e.g. \cite[Thm 6.8]{Matilla}). On the other hand, there is a distinct genre of results that challenges this intuition. Keleti \cite{KeletiDimOneSet}  constructs a set $X \subset \RR$ that avoids all solutions of the equation $x_2 - x_1 = x_4 - x_3$ with $x_1 < x_2 \leq x_3 < x_4$, and which consequently does not contain any nontrivial arithmetic progression. Maga \cite{Maga} constructs a set $X \subset \RR^2$ of full Hausdorff dimension such that no four points in $X$ form the vertices of a parallelogram. The pattern avoidance problem (informally stated) asks: for a given pattern, how large can the dimension of a set $X\subset\RR^d$ be before it is forced to contain a copy of this pattern? 

One way to formalize the notion of a pattern is as follows. If $d\geq 1$ and $n\geq 2$ are integers, we define a pattern to be a set $Z\subset \RR^{dn}$. We say that a set $X\subset\RR^d$ avoids the pattern $Z$ if for every $n$-tuple of distinct points $x_1,\ldots,x_n\in X$, we have $(x_1,\ldots,x_n)\not\in Z$. For example, a set $X\subset\RR^d$ avoids the pattern 
$$
Z = \{(x_1,x_2,x_3)\in\RR^{3d}\colon |(x_1-x_2)\wedge (x_1-x_3)|=0\}
$$ 
if and only if it does not contain three collinear points. Here $u \wedge v$ denotes the wedge product of $u$ and $v$; its length specifies the area of the parallelogram with sides $u$ and $v$. This length vanishes if and only if $u$ and $v$ are parallel. Similarly, a set $X\subset\RR^2$ avoids the pattern 
$$
Z=\{(x_1,x_2,x_3,x_4)\in\RR^{8}\colon |(x_1-x_2)\wedge (x_3-x_4)|=0, \quad |x_1 - x_2| = |x_3 - x_4| \}
$$ 
if and only if no four points in $X$ form the vertices of a parallelogram. 

A number of recent articles have established pattern avoidance results for increasingly general patterns. In \cite{Mathe}, M\'{a}th\'{e} constructs a set $X\subset\RR^d$ that avoids a pattern specified by a countable union of algebraic varieties of controlled degree. In \cite{MalabikaRob}, Fraser and the second author consider the pattern avoidance problem for countable unions of $C^1$ manifolds.

In this paper, we consider the pattern avoidance problem for `rough' patterns $Z\subset\RR^{dn}$ that are the countable union of sets with controlled lower Minkowski dimension. 

\begin{theorem}\label{mainTheorem}
	Let $d \leq \alpha < dn$ and let $Z \subset \RR^{dn}$ be a countable union of compact sets, each with lower Minkowski dimension at most $\alpha$. Then there exists a set $X \subset [0,1)^d$ with Hausdorff dimension at least $(nd - \alpha)/(n-1)$ such that whenever $x_1, \dots, x_n \in X$ are distinct, we have $(x_1, \dots, x_n) \not \in Z$.
\end{theorem}

{\em{Remarks: }} 
\begin{enumerate}[1.]
\item	When $\alpha < d$, the pattern avoidance problem is trivial, since $X = [0,1)^d - \pi(Z)$ is full dimensional and solves the pattern avoidance problem, where $\pi(x_1, \dots, x_n) = x_1$ is a projection map from $\RR^{dn}$ to $\RR^d$. The case $\alpha = dn$ is trivial as well, since we can set $X = \emptyset$.

\item When $Z$ is a countable union of smooth manifolds in $\mathbb R^{nd}$ of co-dimension $m$, we have $\alpha = nd - m$. In this case Theorem \ref{mainTheorem} yields a set in $\mathbb R^d$ of dimension $(nd - \alpha)/(n-1) = m/(n-1)$. This recovers Theorem 1.1 and 1.2 from \cite{MalabikaRob}, making Theorem \ref{mainTheorem} a generalization of the same. Since Theorem \ref{mainTheorem} does not require any regularity assumptions on the set $Z$, the current paper considers the pattern avoidance problem in contexts that cannot be addressed using previous methods. Two such applications, new to the best of our knowledge, have been recorded in Section \ref{applications}; see Theorems \ref{sumset-application} and \ref{C1IsoscelesThm} there.

\item The set $X$ in Theorem \ref{mainTheorem} is obtained by constructing a sequence of approximations to $X$, each of which avoids the pattern $Z$ at different scales. For a sequence of lengths $l_k \searrow 0$, we construct a nested family of sets $\{X_k\}$, where $X_k$ is a union of cubes of sidelength $l_k$ that avoids $Z$ at scales close to $l_n$. The set $X=\bigcap X_k$ avoids $Z$ at all scales. While this proof strategy is not new, our method for constructing the sets $\{X_k\}$ has several innovations that simplify the analysis of the resulting set $X=\bigcap X_k$. In particular, through a probabilistic selection process we are able to avoid the complicated queuing techniques used in \cite{KeletiDimOneSet} and \cite{MalabikaRob}, that required storage of data from each step of the iterated construction, to be retrieved at a much later stage of the construction process.

At the same time, our construction continues to share certain features with \cite{MalabikaRob}. For example, between each pair of scales $l_{k-1}$ and $l_{k}$, we carefully select an intermediate scale $r_{k}$. The set $X_{k}\subset X_{k-1}$ avoids $Z$ at scale $l_{k}$, and it is `evenly distributed' at scale $r_k$: the set $X_{k}$ is a union of intervals of length $l_{k}$ whose midpoints resemble (a large subset of) an arithmetic progression of step size $r_k$. The details of a single step of this construction are described in Section \ref{discretesection}. In Section \ref{discretizationsection}, we explain how the length scales $l_k$ and $r_k$ for $X$ are chosen, and prove its avoidance property. In Section \ref{dimensionsection} we analyze the size of $X$ and show that it satisfies the conclusions of Theorem \ref{mainTheorem}.

\end{enumerate}

\section{Frequently Used Notation and Terminology}\label{notationSection}


\begin{enumerate}
	\item A {\it dyadic length} is a number $l$ of the form $2^{-k}$ for some non-negative integer $k$.

	\item Given a length $l > 0$, we let $\B^d_l$ denote the set of all half open cubes in $\RR^d$ with sidelength $l$ and corners on the lattice $(l \cdot \ZZ)^d$, i.e.
	\[ \B^d_l = \{ [a_1, a_1 + l] \times \cdots \times [a_d, a_d+l] : a_k \in l \cdot \ZZ \}. \]
	If $E \subset \RR^d$, we let $\B^d_l(E)$ denote the set of cubes in $\B^d_l$ intersecting $E$, i.e.
	\[ \B^d_l(E) = \{ I \in \B^d_l: I \cap E = \emptyset \}. \]

	\item\label{defnMinkowskiDim} The {\it lower} and {\it upper Minkowski dimension} of a compact set $Z \subset \RR^d$ are defined as
	%
	\[		\lowminkdim(Z) = \liminf_{l \to 0} \frac{\log(\# \B^d_l(Z))}{\log(1/l)}\quad \text{and}\quad \upminkdim(Z) = \limsup_{l \to 0} \frac{\log(\# \B^d_l(Z))}{\log(1/l)}. \]

	\item If $0 \leq \alpha$ and $\delta > 0$, we define the dyadic Hausdorff content of a set $E\subset\RR^d$ as 
	\[ H^\alpha_\delta(E) = \inf \left\{ \sum_{k = 1}^m l_k^\alpha : E \subset \bigcup_{k = 1}^m I_k\ \text{and}\ I_k \in \B^d_{l_k}, l_k \leq \delta\ \text{for all $k$} \right\}. \]
	The $\alpha$-dimensional dyadic Hausdorff measure $H^\alpha$ on $\RR^d$ is $H^\alpha(E) = \lim_{\delta \to 0} H_\delta^\alpha(E)$, and the {\it Hausdorff dimension} of a set $E$ is $\hausdim(E) = \inf \{ \alpha \geq 0 : H^\alpha(E) = 0 \}$.

	\item \label{stronglyNonDiagonalDef}Given $I \in \B^{dn}_l$, we can decompose $I$ as $I_1 \times \cdots \times I_n$ for unique cubes $I_1, \dots, I_n \in \B_l^d$. We say $I$ is {\it strongly non-diagonal} if the cubes $I_1, \dots, I_n$ are distinct. Strongly non-diagonal cubes will play an important role in Section \ref{discretesection}, when we solve a discrete version of Theorem \ref{mainTheorem}.

	\item\label{strongCoverDefn} Adopting the terminology of \cite{KatzTao}, we say a collection of sets $\{ U_k \}$ is a {\it strong cover} of a set $E$ if $E \subset \limsup U_k$, which means every element of $E$ is contained in infinitely many of the sets $U_k$. This idea will be useful in Section \ref{discretizationsection}.  

	\item\label{frostmanItem} A {\it Frostman measure} of dimension $\alpha$ is a non-zero compactly supported probability measure $\mu$ on $\RR^d$ such that for every cube $I$ of sidelength $l$, $\mu(I) \lesssim l^\alpha$. Note that a measure $\mu$ satisfies this inequality for every cube $I$ if and only if it satisfies the inequality for cubes whose sidelengths are dyadic lengths. {\it Frostman's lemma} says that
	\begin{equation}  \hausdim(E) = \sup \left\{ \alpha: 
	\begin{aligned} 
	&\text{there is a Frostman measure of}\\
	&\text{dimension $\alpha$ supported on $E$} 
	\end{aligned} 
	\right\}.  \label{Hdim-defn} \end{equation} 
\end{enumerate}

\section{Avoidance at Discrete Scales}\label{discretesection}

In this section we describe a method for avoiding $Z$ at a single scale. We apply this technique in Section \ref{discretizationsection} at many scales to construct a set $X$ avoiding $Z$ at all scales. This single scale avoidance technique is the building block of our construction, and the efficiency with which we can avoid $Z$ at a single scale has direct consequences on the Hausdorff dimension of the set $X$ obtained in Theorem \ref{mainTheorem}.

At a single scale, we solve a discretized version of the problem, where all sets are unions of cubes at two dyadic lengths $l \geq s$ (later, we will choose $l=l_n$ and $s=l_{n+1}$). Given a set $E \subseteq [0,1)^d$ that is a union of cubes in $\B_l^d$, our goal is to construct a set $F\subset E$ that is a union of cubes in $\B_s^d$ such that $F^n$ is disjoint from the strongly non-diagonal cubes of $\B_{s}^{dn}(Z)$.

In order to ensure the final set $X$ obtained in Theorem \ref{mainTheorem} has large Hausdorff dimension regardless of the rapid decay of scales used in the construction of $X$, it is crucial that $F$ is uniformly distributed at intermediate scales between $l$ and $s$. This is the `non-concentration' property discussed below. The next lemma constructs a set $F$ with these properties. 

\begin{lemma} \label{discretelemma}
	Fix two distinct dyadic lengths $l$ and $s$, $l > s$. Let $E \subseteq [0, 1)^d$ be a nonempty union of cubes in $\B^d_l$, and let $G\subset\RR^{dn}$ be a nonempty union of cubes in $\B_s^{dn}$ such that 
	\begin{equation}\label{ZsLarge}
	(l/s)^d \leq \# \B^{dn}_s(G)  \leq \frac{1}{2}(l/s)^{dn}.
	\end{equation} 
	Then there exists a dyadic length $r \in [s,l]$ of size
	\begin{equation} \label{rBound}
	 	r \sim \big(l^{-d}s^{dn} \# \B^{dn}_s(G)\big)^{\frac{1}{d(n-1)}},
	 \end{equation}
	 and a set $F \subset E$ that is a nonempty union of cubes in $\B^d_s(E)$ satisfying the following three properties:
	\begin{enumerate}
		\item\label{avoidanceItem} \emph{Avoidance}: For any choice of distinct cubes $J_1, \dots, J_n \in \B^d_s(F)$, $J_1 \times \dots \times J_n \not \in \B_s^{dn}(G)$.
		\item\label{nonConcentrationItem} \emph{Non-Concentration}: For every $I' \in \B_r^d(E)$, there is at most one $J \in \B_s^d(F)$ with $J \subset I'$.
		\item\label{largeSizeItem} \emph{Large Size}: For every $I \in \B^d_l(E)$, $\# \B^d_s(F \cap I) \geq \# \B^d_r(I) / 2 = (l/r)^d / 2$.
	\end{enumerate}
	\end{lemma}
	{\em{Remark: }} Item \ref{avoidanceItem} says that $F$ avoids strongly non-diagonal cubes in $\B^{dn}_s(G)$. Items \ref{nonConcentrationItem} and \ref{largeSizeItem} together imply that for every $I \in \B^d_l(E)$, at least half the cubes $I'\in \B_r^d(I)$ contribute a single sub-cube of sidelength $s$ to $F$; the rest contribute none.  

\begin{proof}
Let $r$ be the smallest dyadic length satisfying 
\begin{equation} \label{What-is-r}
r\geq R := \big(2 l^{-d}s^{dn}\# \B^{dn}_s(G)\big)^{\frac{1}{d(n-1)}}.
\end{equation} 
This choice of $r$ satisfies \eqref{rBound}. 
	%
	The inequalities in \eqref{ZsLarge} ensure that $r \in [s,l]$; more precisely, the left inequality in \eqref{ZsLarge} implies $R$ is bounded from below by $s$, proving $r \geq s$. On the other hand, the right inequality in \eqref{ZsLarge} shows that $R$ bounded from above by $l$. The minimality of $r$ then proves that $r \leq l$.  

	For each $I' \in \B_r^d(E)$, let $J_{I'}$ be an element of $\B^d_s(I)$ chosen uniformly at random; these choices are independent as $I'$ ranges over the elements of $\B_r^d(E)$. Define
	\[ 	U = \bigcup \left\{ J_{I'} :  I' \in \B_r^d(E) \right\}, \]
	and
	\[ \mathcal{K}(U) = \{ K \in \B^{dn}_s(G) : K \in U^n, \text{$K$ strongly non-diagonal} \}. \]
	Note that the sets $U$ and $\mathcal{K}(U)$ are random sets, in the sense that they are depend on the random variables $\{J_I\}$. 
	Define
	\begin{equation} \label{defnOfF}
		F_U = U - \{ \pi(K): K \in \mathcal{K}(U) \},
	\end{equation}
	where $\pi: \RR^{dn} \to \RR^d$ is the projection map $(x_1, \cdots, x_n) \in \mathbb R^{dn} \mapsto x_1 \in \mathbb R^d$. Thus $\pi$ sends the cube $K_1 \times \dots \times K_n\in \B^{dn}_s$ to the cube $K_1\in \B^{d}_s$. Given any strongly non-diagonal cube $J_1 \times \cdots \times J_n \in \B_s^{dn}(G)$, either $J_1 \times \cdots \times J_n \not \in \B_s^{dn}(U^n)$, or $J_1 \times \cdots \times J_n \in \B_s^{dn}(U^n)$. If the former occurs then $J_1 \times \cdots \times J_n \not \in \B_s^{dn}(F_U^n)$ since $F_U \subset U$, while if the latter occurs then $K \in \mathcal{K}(U)$, so $J_1 \not \in \B_s^d(F_U)$. In either case, $J_1 \times \cdots \times J_n \not \in \B_s^{dn}(F_U^n)$, so $F_U$ satisfies Property \ref{avoidanceItem}. By construction, $U$ contains at most one subcube $J \in \B^{dn}_s$ for each $I \in \B^{dn}_l(E)$. Since $F_U \subset U$, $F_U$ satisfies Property \ref{nonConcentrationItem}. Thus the set $F_U$ satisfies Properties \ref{avoidanceItem} and \ref{nonConcentrationItem} regardless of which values are assumed by the random variables $\{J_I\}$. Next we will show that with non-zero probability, the set $F_U$ satisfies Property \ref{largeSizeItem}. 

	For each cube $J \in \B_s^d(E)$, there is a unique `parent' cube $I' \in \B_r^d(E)$ such that $J \subset I'$. Since $I'$ contains $(r/s)^d$ elements of $\B^d_s(E)$, and $J_{I'}$ is chosen uniformly at random from $\B^d_s(I)$,
	%
	\[ \Prob(J \subset U) = \Prob(J_{I'} = J) = (s/r)^d. \]
	%
	The cubes $J_I$ are chosen independently, so if $J_1, \dots, J_k$ are distinct cubes in $\B^d_s(E)$, then 
	\begin{equation}\label{jointprob}
	\Prob(J_1, \dots, J_k \in U) = \begin{cases} (s/r)^{dk} & \text{if $J_1, \dots, J_k$ have distinct parents,} \\ 0 & \text{otherwise}. \end{cases} 
	\end{equation}
	Let $K = J_1 \times \dots \times J_n \in \B^{dn}_s(G)$. If the cubes $J_1, J_2, \cdots J_n$ are distinct, we deduce from \eqref{jointprob} that
	\begin{equation}\label{probaKSubsetUn}
		\Prob(K \subset U^n) = \Prob(J_1, \dots, J_n \in U) = (s/r)^{dn}.
	\end{equation}
	By \eqref{probaKSubsetUn} and linearity of expectation, we have 
	\begin{align*}
		\Expect(\# \mathcal{K}(U)) &= \sum_{K \in \B^{dn}_s(G)} \Prob(K \subset U^n) \\
		& \leq \# \B_s^{dn}(G) \cdot (s/r)^{dn} \\
		&\leq \left[ l^d/2 \right] r^{-d},
	\end{align*}
	where the last inequality is just a restatement of \eqref{What-is-r}. 
	In particular, there exists at least one (non-random) set $U_0$ such that
	\begin{equation}\label{KU0Small}
		\# \mathcal{K}(U_0) \leq \Expect(\# \mathcal{K}(U)) \leq \frac{1}{2}(l/r)^d.
	\end{equation}
	 In other words, $F_{U_0}\subset U_0$ is obtained by removing at most $\frac{1}{2}(l/r)^d$ cubes in $\B^d_s$ from $U_0$. For each $I \in \B_l^d(E)$, we know that 
\[ \# \B_{s}^d(I \cap U_0) = (l/r)^d. \] 	
Combining this with the previous observation, we arrive at the estimate 
\begin{align*}  \# \B_{s}^d(I \cap F_{U_0}) &= \# \B_{s}^d(I \cap F_{U_0}) - \# \B_{s}^d \bigl[ I \cap \pi(\mathcal K(U_0)) \bigr] \\  
&\geq \# \B_{s}^d(I \cap F_{U_0}) - \# \B_{s}^d (\pi(\mathcal K(U_0))) \\ 
& \geq  \# \B_{s}^d(I \cap F_{U_0}) - \# \B_{s}^d (\mathcal K(U_0)) \geq \frac{1}{2} (l/r)^d.  
\end{align*}  
	In other words, $F_{U_0}$ satisfies Property \ref{largeSizeItem}. Setting $F = F_{U_0}$ completes the proof.
\end{proof}

{\em{Remarks: }} 
\begin{enumerate}[1.]
\item	While Lemma \ref{discretelemma} uses probabilistic arguments, the conclusion of the lemma is not a probabilistic statement. In particular, one can find a suitable $F$ constructively by checking every possible choice of $U$ (there are finitely many) to find one particular choice $U_0$ which satisfies \eqref{KU0Small}, and then defining $F$ by \eqref{defnOfF}. Thus the set we obtain in Theorem \ref{mainTheorem} exists by purely constructive means.

\item At this point, it is possible to understand the numerology behind the Hausdorff dimension bound $\dim_H(X) \geq (nd-\alpha)/(n-1)$ from Theorem \ref{mainTheorem}. We will pause to do so here before returning to the proof of Theorem \ref{mainTheorem}. For simplicity, suppose that $Z\subset\RR^{dn}$ satisfies 
\begin{equation}\label{specialCase}
\#\B_{s}^{dn}(Z)\leq Cs^{-\alpha}\quad  \textrm{for every}\ s\in(0,1],
\end{equation}
where $C>0$ is a fixed constant. Let $l=1$, let $E=[0,1)^d$, let $s>0$ be a small parameter. If $s$ is chosen sufficiently small compared to $d,n,\alpha,$ and $C$, then \eqref{ZsLarge} is satisfied, and we can apply Lemma \ref{discretelemma} to find a dyadic scale $r\sim s^{\frac{dn-\alpha}{d(n-1)}}$ and a set $F$ that avoids the strongly non-diagonal cubes of $B_{s}^{dn}(Z)$. The set $F$ is a union of approximately $r^{-d}\sim s^{-\frac{dn-\alpha}{n-1}}$ cubes of sidelength $s$. Informally, the set $F$ resembles a set with Minkowski dimension $\alpha$ when viewed at scale $s$. 

The set $X$ constructed in Theorem \ref{mainTheorem} will be obtained by applying Lemma \ref{discretelemma} iteratively at many scales. At each of these scales, $X$ will resemble a set of Minkowski dimension $\frac{dn-\alpha}{n-1}$. A careful analysis of the construction (performed in Section \ref{dimensionsection}) shows that $X$ actually has Hausdorff dimension at least $\frac{dn-\alpha}{n-1}$.

\item	Lemma \ref{discretelemma} is the core method in our avoidance technique. The remaining argument is fairly modular. If, for a special case of $Z$, one can improve the result of Lemma \ref{discretelemma} so that $r$ is chosen on the order of $s^{\beta/d}$, then the remaining parts of our paper can be applied near verbatim to yield a set $X$ with Hausdorff dimension $\beta$, as in Theorem \ref{mainTheorem}. 

\end{enumerate}

\section{Fractal Discretization}\label{discretizationsection}
In this section we will construct the set $X$ from Theorem \ref{mainTheorem} by applying Lemma \ref{discretelemma} at many scales. The goal is to find a nested decreasing family of discretized sets $\{ X_k \}$ and to set $X = \bigcap X_k$. One condition guaranteeing that $X$ avoids $Z$ is that $X_k^n$ is disjoint from {\it strongly non-diagonal} cubes in $Z_k$.

\begin{lemma} \label{stronglydiagonal}
	Let $Z \subset \RR^{dn}$, let $\{Z_k\}$ be a sequence of sets that strongly cover $Z$, and let $\{ l_k \}$ be a sequence of lengths converging to zero. For each index $k$, let $X_k$ be a union of cubes in $\B^d_{l_k}$. Suppose that for each $k$, $X_k^n$ avoids strongly non-diagonal cubes in $\B^{dn}_{l_k}(Z_k)$. Then for any distinct $x_1, \dots, x_n \in \bigcap X_k$, we have $(x_1, \dots, x_n) \not \in Z$.
\end{lemma}
\begin{proof}
	Let $z \in Z$ be a point with distinct coordinates $z_1, \dots, z_n$. Define
	\[ \Delta = \{ (w_1, \dots, w_n) \in \RR^{dn}: \text{there exists $i \neq j$ such that $w_i = w_j$} \}. \]
	Then $d(\Delta,z) > 0$, where $d$ is the Hausdorff distance between $\Delta$ and $z$. Since $\{ Z_k \}$ strongly covers $Z$, there is a subsequence $\{ k_m \}$ such that $z \in Z_{k_m}$ for every index $m$. Since $l_k\searrow 0$ and thus $l_{k_m}\searrow 0$, if $m$ is sufficiently large then $\sqrt{dn}l_{k_m}<\Delta$ (note that $\sqrt{dn}l_{k_m}$ is the diameter of a cube in $\B_{l_{k_m}}^{dn}$). For such a choice of $m$, we have that if $I\in \B_{l_{k_m}}^{dn}(Z_{k_m})$ is the (unique) cube in $\B_{l_{k_m}}^{dn}$ containing $z$, then $I\cap\Delta=\emptyset$. But this means $I$ is strongly non-diagonal. Since $X_{k_m}$ avoids the strongly non-diagonal cubes of $Z_{k_m}$, we conclude that $z \not \in X_{k_m}^n$. In particular, $z\not\in \bigcap_{k=1}^\infty X_k$.
\end{proof}

We are now ready to construct the set $X$ in Theorem \ref{mainTheorem}. Recall that $Z \subset \RR^{dn}$ is a countable union of compact sets, each with lower Minkowski dimension at most $\alpha$. Thus we can write \[ Z =  \bigcup_{i=1}^{\infty} Y_i  \quad \text{ with }  \lowminkdim(Y_i) \leq \alpha  \text{ for each $i$}. \]  Recall that in the statement of Theorem \ref{mainTheorem}, we assumed that $\alpha\geq d$. However, it might be the case that some of the sets $Y_i$ have lower Minkowski dimension smaller than $d$. We will deal with this minor annoyance as follows. Let $H\subset[0,1)^{dn}$ be a set satisfying $\#\B_{l}^{dn}(H)\geq l^{-d}$ for each $l\in (0,1]$ (for example, $H$ could be the intersection of $[0,1)^{dn}$ with a finite union of $d$-dimensional hyperplanes). Let $\{i_k\}$ be a sequence of integers that repeats each integer infinitely often and define $Z_k = Y_{i_k}\cup H.$ The sequence of sets $\{Z_k\}$ is a strong cover of $Z$; each set $Z_k$ has lower Minkowski dimension at most $\alpha$ and satisfies the bound 
\begin{equation}\label{lowerBoundZk}
\#\B_{l}^{dn}(Z_k)\geq l^{-d}\quad\textrm{for all}\ l\in (0,1].
\end{equation}

Our next task is to specify the length scales $\{l_k\}$. We define these scales inductively, predicated on a sequence of small constants $\epsilon_k \searrow 0$ that is fixed at the outset. We will choose the sequence $\epsilon_k \searrow 0$ so that $dn-\alpha-2\epsilon_k>0$ for each $k$. Define $l_0=1$. Suppose that the length scales $l_0,\ldots,l_{k-1}$ have been chosen. Define
Since $\lowminkdim(Z_k) \leq \alpha$, Definition \ref{defnMinkowskiDim} implies that there exist arbitrarily small lengths $l$ which satisfy 
\begin{equation}\label{coveringOfBdnlZk}
\# \B^{dn}_l(Z_k) \leq l^{-\alpha - \frac{\varepsilon_k}{2}}.
\end{equation}
In addition, we can choose $l = l_k$ small enough to satisfy 
\begin{align}
l^{dn-\alpha-\varepsilon_k} & \leq \frac{1}{2}l_{k-1}^{dn}, \text{ and } \label{coverBoundRequirement}\\
l^{\epsilon_k} & \leq l_{k-1}^{2d}\label{quadDecayRequirement}.
\end{align}

\subsection{Construction of $X$} 
\begin{lemma} 
Given a sequence of dyadic length scales $l_k$ obeying \eqref{coveringOfBdnlZk}, \eqref{coverBoundRequirement} and \eqref{quadDecayRequirement} as above, there exists a sequence of sets $\{X_k\}$ and a sequence of dyadic intermediate scales $r_k$ obeying the following properties. Each set $X_k$ is a union of cubes in $\B_{l_k}^d(X_{k-1})$ that avoids the strongly non-diagonal cubes of $\mathcal B_{l_k}^{dn}(Z_k)$. Furthermore, for each index $k\geq 1$ we have
\begin{align}
& l_k\leq r_k\leq l_{k-1},\\
& r_k \lesssim l_{k}^{\frac{dn-\alpha -\epsilon_k}{d(n-1)}},\label{rkSizeBound}\\
& \# \B^d_{l_k}(X_k \cap I) \geq \frac{1}{2}(l_{k-1}/r_k)^d  \text{ for each}\ I\in \B_{l_{k-1}}^d(X_{k-1}), \label{manyIkInIkm1}\\
&\# \B^d_{l_k}(X_k \cap I') \leq 1 \text{ for each}\ I' \in \B_{r_{k}}^d(X_{k-1}).\label{XkWellDistributed}
\end{align}
\end{lemma} 
\begin{proof} 
We will proceed to construct $X_k$ by induction, using Lemma \ref{discretelemma} as building block. Set $X_0=[0,1)^d$.  
Next, suppose that the sets $X_0,\ldots,X_{k-1}$ have been defined. Our goal is to apply Lemma \ref{discretelemma} to $E=X_{k-1}$ and $G=Z_k$ with $l=l_{k-1}$ and $s=l_k$. 
This will be possible once we verify the hypothesis \eqref{ZsLarge}, which in this case takes the form 
\begin{equation}  \left(l_{k-1}/l_k \right)^{d} \leq \#\B_{l_k}^{dn}(Z_k) \leq \frac{1}{2} \left(l_{k-1}/l_k \right)^{dn}. \label{need-to-check} \end{equation} 
The right hand side follows from inequalities \eqref{coveringOfBdnlZk} and \eqref{coverBoundRequirement}. 
On the other hand, \eqref{lowerBoundZk} and that fact that $l_k, l_{k-1}\leq 1$ implies that
$$
(l_{k-1}/l_k)^d\leq l_{k}^{-d}\leq \#\B_{l_k}^{dn}(Z_k), 
$$
establishing the left inequality in \eqref{need-to-check}. Applying Lemma \ref{discretelemma} as described above now produces a dyadic length 
\begin{equation}\label{definOfr}
r\sim \big(l_{k-1}^{-d}l_k^{dn} \# \B^{dn}_{l_k}(Z_k)\big)^{\frac{1}{d(n-1)}} 
\end{equation}
and a set $F\subset X_{k-1}$ that is a union of cubes in $\B_{l_k}^{d}$. The set $F$ satisfies Properties \ref{avoidanceItem}, \ref{nonConcentrationItem}, and \ref{largeSizeItem} from the statement of Lemma \ref{discretelemma}. Define $r_k=r$ and $X_k=F$. The estimate  \eqref{rkSizeBound} on $r_k$ follows from \eqref{definOfr} using the known bounds \eqref{coveringOfBdnlZk} and \eqref{quadDecayRequirement}:  
\[ r_k \lesssim \bigl( l_{k-1}^{-d}  l_k^{dn -\alpha - \frac{\epsilon_k}{2}} \bigr)^{\frac{1}{d(n-1)}} =  \bigl( l_{k-1}^{-d} l_k^{\frac{\epsilon_k}{2}} l_k^{dn -\alpha - \epsilon_k} \bigr)^{\frac{1}{d(n-1)}} = \bigl( l_{k-1}^{-2d} l_k^{\epsilon_k}\bigr)^{\frac{1}{2d(n-1)}} l_{k}^{\frac{dn-\alpha -\epsilon_k}{d(n-1)}} \lesssim l_{k}^{\frac{dn-\alpha -\epsilon_k}{d(n-1)}}. \] 
The requirements \eqref{manyIkInIkm1} and \eqref{XkWellDistributed} follow from Properties \ref{largeSizeItem} and \ref{nonConcentrationItem} of Lemma \ref{discretelemma} respectively.
\end{proof} 

Now that we have defined the sets $\{X_k\}$, we set $X:=\bigcap X_k$. Since $X_k$ avoids strongly non-diagonal cubes in $Z_{k}$, Lemma \ref{stronglydiagonal} implies that if $x_1, \dots, x_n \in X$ are distinct, then $(x_1, \dots, x_n) \not \in Z$. To finish the proof of Theorem \ref{mainTheorem}, we must show that $\dim_{\mathbf H}(X)\geq \frac{dn - \alpha}{n - 1}$. This will be done in the next section.

\section{Dimension Bounds}\label{dimensionsection}
To complete the proof of Theorem \ref{mainTheorem}, we must show that $\dim_{\mathbf{H}}(X) \geq  \frac{dn - \alpha}{n - 1}$.  
%
In view of \eqref{Hdim-defn}, we will do this by constructing a Frostman measure of appropriate dimension supported on $X$. 
%
%

We start by defining a premeasure on $\bigcup_{i = 1}^\infty \B^d_{l_i}[0,1)^d$. Set $\mu([0,1)^d) = 1$. Suppose now that $\mu(I)$ has been defined for all cubes in $\bigcup_{i = 1}^{k-1} \B^d_{l_i}[0,1)^d$, and let $J \in \B^d_{l_k}$. Let $I \in \B^d_{l_{k-1}}$ be the `parent cube' of $J$ (i.e. $I$ is the unique cube in $\B^d_{l_{k-1}}$ with $J \subset I$). Define
\begin{equation} \label{muRecurse} 
\mu(J) = \left\{ \begin{array}{ll}
{\mu(I)}/{\# \B^d_{l_k}(X_k \cap I)}& \textrm{if }  J \subset X_k,\\
0& \textrm{otherwise}.
\end{array}
\right. 
\end{equation}
Observe that for each index $k\geq 1$ and each $I \in \B_{l_{k-1}}^d$, 
\begin{equation}\label{muBreakDown}
\sum_{J\in \B_{l_k}^d(I) }\mu(J)=\sum_{J \in \B_{l_k}^d(X_k\cap I) }\mu(J) = \mu(I).
\end{equation}
In particular, for each index $k$ we have
$$
\sum_{I\in\B_{l_k}}\mu(I)=1.
$$
By a standard argument involving the Caratheodory extension theorem \cite[Proposition 1.7]{Falconer}, the premeasure $\mu$ extends to a measure on the Borel subsets of $[0,1)^d$. Note that for each $k\geq 1$, $\operatorname{supp}(\mu)\subset X_k$. Thus $\mu$ is supported on $\bigcap X_k = X$. To complete the proof of Theorem \ref{mainTheorem} we will show that $\mu$ is a Frostman measure of dimension $\frac{dn - \alpha}{n - 1}-\epsilon$ for every $\epsilon>0$.

\begin{lemma}\label{massSomeScales}
	For each $k\geq 1$ and each $J \in \B^d_{l_k}(X)$, 
	$$
	\mu(J) \lesssim l_k^{\frac{dn-\alpha}{n-1}- \eta_k}, \quad \text{ where } \quad \eta_k = \frac{n+1}{2(n-1)} \epsilon_k \searrow 0 \text{ as } k \rightarrow \infty.
	$$
\end{lemma}
\begin{proof}
	Let $J \in \B^d_{l_k}$ and let $I \in B^d_{l_{k-1}}$ be the parent of cube of $I$. Since $\mu$ is a probability measure, we have $\mu(I^\prime)\leq 1$. Combining \eqref{muRecurse}, \eqref{manyIkInIkm1}, \eqref{rkSizeBound}, and \eqref{quadDecayRequirement} we obtain
	$$
	\mu(J)\leq \frac{2r_k^d}{l_{k-1}^d}\mu(I)\leq \frac{2r_k^d}{l_{k-1}^d}\lesssim \frac{l_{k}^{\frac{dn-\alpha - \epsilon_k}{n-1}}}{l_{k-1}^d}=l_k^{\frac{dn-\alpha}{n-1}-\eta_k}\big(l_k^{\frac{\epsilon_k}{2}}/l_{k-1}^d\big)\leq l_k^{\frac{dn-\alpha}{n-1}-\eta_k}.\qedhere
	$$
\end{proof}

\begin{corollary}\label{muAtScaleRk}
For each $k\geq 1$ and each $I' \in \B^d_{r_k} (X_{k-1})$, 
	\begin{equation} 
	\mu(I') \lesssim (r_k/l_{k-1})^d l_{k-1}^{\frac{dn-\alpha}{n-1}-\eta_{k-1}}. \label{mu-Rk}
	\end{equation} 
\end{corollary}
\begin{proof}
Let us fix a cube $I' \in \B^d_{r_k}(X_{k-1})$, and let $I$ denote its unique parent cube in $\B_{l_{k-1}}^d (X_{k-1})$. According to \eqref{XkWellDistributed}, $I'$ contains at most one cube in $\B_{l_k}^d(I)$; let us denote this cube by $J(I')$ if it exists. Then the mass distribution rule given by \eqref{muRecurse} dictates that 
\begin{align*}
\mu(I') = \mu(X_k \cap I') = \begin{cases} \mu(J(I')) = {\mu(I)}/{\# \B_{l_k}^d(X_k \cap I)}  &\text{ if } \# \B_{l_k}^d(X_k \cap I) = 1, \\ 0 &\text{ otherwise.} \end{cases} 
\end{align*} 
Using the estimate \eqref{manyIkInIkm1} and applying Lemma \ref{massSomeScales} to $I \in \mathcal B_{l_{k-1}}^d(X)$, we arrive at the claimed bound \eqref{mu-Rk}. 
\end{proof}
Lemma \ref{massSomeScales} and Corollary \ref{muAtScaleRk} allow us to control the behavior of $\mu$ at all scales. 

\begin{lemma} \label{frostmanBound}
For every $\alpha \in [d, dn)$, and for each $\epsilon>0$, there is a constant $C_\epsilon$ so that for all dyadic lengths $l\in (0,1]$ and all $I \in \B_l^d$, we have
	\begin{equation} 
	\mu(I) \leq C_{\epsilon} l^{\frac{dn - \alpha}{n - 1} - \epsilon}. \label{mu-ball-condition} 
	\end{equation} 
\end{lemma}
\begin{proof} Fix $\epsilon > 0$. Since $\eta_k \searrow 0$ as $k\to\infty$, there is a constant $C_{\epsilon}$ so that $l_k^{-\eta_k}\leq C_{\epsilon}l_k^{-\epsilon}$ for each $k\geq 1$ (for example, we could choose $C_{\epsilon}=l_{k_0}^{-\eta_{k_0}}$, where $k_0$ is the largest integer for which $\eta_{k_0}<\epsilon$). Next, let $k$ be the (unique) index so that $l_{k+1}\leq l\leq l_{k}$. We will split the proof of \eqref{mu-ball-condition} into two cases, depending on the position of  $l$ within $[l_{k+1}, l_k]$. 

	{\em{Case 1: }} If $r_{k+1} \leq l \leq l_k$, 
	we can cover $I$ by $(l/r_{k+1})^d$ cubes in $\B^d_{r_{k+1}}$. By Corollary \ref{muAtScaleRk},
	\begin{equation}
	\begin{split}
	\mu(I) & \lesssim (l/r_{k+1})^d (r_{k+1}/l_k)^d l_k^{\frac{dn-\alpha}{n-1}-\eta_k} \\
	& =(l/l_k)^d l_k^{\frac{dn-\alpha}{n-1}-\eta_{k}}\\
	& \leq l^{\frac{dn-\alpha}{n-1}} (l/l_k)^{\frac{\alpha - d}{n-1}} l_k^{-\eta_k}\\
	& \leq l^{\frac{dn-\alpha}{n-1} - \eta_k}  \\
	& \leq C_{\epsilon}l^{\frac{dn-\alpha}{n-1}-\epsilon}.
	\end{split}
	\end{equation}
The penultimate inequality is a consequence of our assumption $\alpha \geq d$. 

	{\em{Case 2: }} If $l_{k+1} \leq l \leq r_{k+1},$ we can cover $I$ by a single cube in $\B^d_{r_{k+1}}$. By \eqref{XkWellDistributed}, each cube in $\B^d_{r_{k+1}}$ contains at most one cube $I_0 \in \B^d_{l_{k+1}}(X_{k+1})$, so by Lemma \ref{massSomeScales},  
		$$ 
		\mu(I) \leq \mu(I_0) \lesssim l_{k+1}^{\frac{dn - \alpha}{n - 1} - \eta_{k+1}} 
		\leq C_{\epsilon}l_{k+1}^{\frac{dn - \alpha}{n - 1} - \epsilon}
		\leq C_{\epsilon}l^{\frac{dn - \alpha}{n - 1} - \epsilon} \leq C_{\epsilon}l^{\frac{dn - \alpha}{n - 1} - \epsilon }.\qedhere
		$$

\end{proof}

Applying Frostman's lemma to Lemma \ref{frostmanBound} gives $\hausdim(X) \geq \frac{dn - \alpha}{n - 1} - \epsilon$ for every $\epsilon>0$, which concludes the proof of Theorem \ref{mainTheorem}.

\section{Applications}\label{applications}

As discussed in the introduction, Theorem \ref{mainTheorem} generalizes Theorems 1.1 and 1.2 from \cite{MalabikaRob}. In this section, we present two applications of Theorem \ref{mainTheorem} in settings where previous methods do not yield any results.

\subsection{Sum-sets avoiding specified sets}

\begin{theorem} \label{sumset-application} 
	Let $Y \subset \RR^d$ be a countable union of sets of Minkowski dimension at most $\beta < d$. Then there exists a set $X \subset \RR^d$ with Hausdorff dimension at least $d - \alpha$ such that $X + X$ is disjoint from $Y$.
\end{theorem}
\begin{proof}
	Define $Z = Z_1 \cup Z_2$, where
	\[ Z_1 = \{ (x,y) : x + y \in Y \} \quad \text{and} \quad Z_2 = \{ (x,y): y \in Y/2 \}. \]
	Since $Y$ is a countable union of sets of Minkowski dimension at most $\beta$, $Z$ is a countable union of sets with lower Minkowski dimension at most $d + \beta$. Applying Theorem \ref{mainTheorem} with $n = 2$ and $\alpha = d + \beta$ produces a set $X \subset \RR^d$ with Hausdorff dimension $2d  - (d + \beta) = d - \beta$ avoiding $Z$. We claim that $X+ X$ is disjoint from $Y$. To see this, first suppose $x, y \in X$, $x \ne y$. Since $X$ avoids $Z_1$, we conclude that $x + y \not \in Y$. Suppose now that $x = y \in X$. Since $X$ avoids $Z_2$, we deduce that $X \cap (Y/2) = \emptyset$, and thus for any $x \in X$, $x + x = 2x \not \in Y$. This completes the proof. 
\end{proof}

\subsection{Subsets of Lipschitz curves avoiding isosceles triangles}

In \cite{MalabikaRob}, Fraser and the second author prove that if $\gamma\subset\RR^n$ is a simple $C^2$ curve with non-vanishing curvature, then there exists a set $S\subset\gamma$ of Hausdorff dimension $1/2$ that does not contain the vertices of an isosceles triangle. Using Theorem \ref{mainTheorem}, we generalize this result to Lipschitz curves. 

\begin{theorem}\label{C1IsoscelesThm}
Let $g\colon [0,1]\to \RR^{n-1}$ be Lipschitz. Then there is a set $X\subset [0,1]$ of Hausdorff dimension $1/2$ so that the set $\{(t,g(t))\colon t\in X\}$ not contain the vertices of an isosceles triangle.
\end{theorem}

\begin{proof}
Choose $M>0$ so that for all $s,t\in [0,1]$, we have $\Vert g(s)-g(t)\Vert \leq M|s-t|,$ where $\Vert \cdot \Vert$ denotes the Euclidean norm in $\RR^{n-1}$. Let $f\colon[0,1]\to\ \RR^{n-1}$ be given by $f(t) = g(\frac{x}{10M})-g(0)$. Then $f$ is $1/10$-Lipschitz and the graph of $f$ is contained in $[0,1]^n$. Define
\begin{align*}
Z=\{(x_1,x_2,x_3)\in [0,1]^3\colon (x_1,f(x_1)),&\ (x_2,f(x_2)),\ (x_3,f(x_3))\ \textrm{form}\\
&\textrm{the vertices of an isosceles triangle}\}.
\end{align*}

We will show that $Z$ has lower Minkowski dimension at most 2. Fix $0<\delta<1$. It suffices to show that
\begin{equation}\label{deltaCoveringZ}
\# \mathcal{B}_{\delta}^{3}(Z)\lesssim\delta^{-2}\log(1/\delta).
\end{equation}

We have
\begin{align*}
\#\mathcal{B}_{\delta}^{3}(Z) &= \sum_{I_1 \in \mathcal{B}_{\delta}^1([0,1])} \#\{ I_2,I_3\in \mathcal{B}_{\delta}^1([0,1])\colon I_1\times I_2\times I_3\in \mathcal{B}_{\delta}^{3}(Z) \}\\
&= \sum_{I_1 \in \mathcal{B}_{\delta}^1([0,1])} \sum_{k=0}^{\log(1/\delta)} \sum_{\substack{I_2 \in \mathcal{B}_{\delta}^1([0,1]) \\ \operatorname{dist}(I_1,I_2)\sim \delta 2^k }} \#\{ I_3\in \mathcal{B}_{\delta}^1([0,1])\colon  I_1\times I_2\times I_3\in \mathcal{B}_{\delta}^{3}(Z) \}.
\end{align*}
In the above expression we abuse notation slightly and say that $\operatorname{dist}(I_1,I_2)\sim \delta$ if $I_1=I_2$; this will not affect our estimates.

Note that for each $I_1 \in \mathcal{B}_{\delta}^1([0,1])$, there are roughly $(\delta 2^k)/\delta=2^k$ intervals $I_2\in  \mathcal{B}_{\delta}^1([0,1])$ with $\operatorname{dist}(I_1,I_2)\sim \delta 2^k$. Thus to establish \eqref{deltaCoveringZ}, it suffices to prove that for each $I_1 \in \mathcal{B}_{\delta}^1([0,1])$ and each $I_2\in  \mathcal{B}_{\delta}^1([0,1])$ with $\operatorname{dist}(I_1,I_2)\sim \delta 2^k$, we have
\begin{equation}\label{numberOfI3}
\#\{ I_3\in \mathcal{B}_{\delta}^1([0,1])\colon  I_1\times I_2\times I_3\in \mathcal{B}_{\delta}^{3}(Z) \}\lesssim 2^{-k}/\delta.
\end{equation}

For each distinct $p,q\in [0,1]^n$, define 
$$
H_{p,q}=\big\{z\in \RR^n\colon \big(z-\frac{p+q}{2}\big)\cdot (p-q)=0  \big\}.
$$
This is the hyperplane passing through the midpoint of $p$ and $q$ that is perpendicular to the line passing through $p$ and $q$. We will call $H_{p,q}$ the perpendicular bisector of $p$ and $q$.

Fix a choice of intervals $I_1$ and $I_2$ with $\operatorname{dist}(I_1,I_2)\sim \delta 2^k$. Let $\tilde I_1$ and $\tilde I_2$ denote the twofold dilates of $I_1$ and $I_2$, respectively. Note that if $I_3\in\mathcal{B}_\delta^1([0,1])$ with $I_1\times I_2\times I_3\in \mathcal{B}_{\delta}^{3}(Z)$, then there are points $x_j\in \tilde I_j,\ i=1,2,3$ so that 
$$
(x_3,f(x_3))\in H_{(x_1,f(x_1)), (x_2,f(x_2))}.
$$

Consider the set
$$
S_{I_1,I_2}=[0,1]^n \cap \bigcup_{\substack{x_1\in \tilde I_1\\ x_2\in \tilde I_2}}H_{(x_1,f(x_1)), (x_2,f(x_2))}.
$$
For each $x_1\in \tilde I_1$ and $x_2\in \tilde I_2$, the line passing through $(x_1,f(x_1))$ and $(x_2,f(x_2))$ makes an angle $\leq 1/10$ with the $e_1$ direction. Thus the hyperplane $H_{(x_1,f(x_1)), (x_2,f(x_2))}$ makes an angle $\leq 1/10$ with the hyperplane spanned by the $e_2,\ldots,e_n$ directions. Since $\tilde I_1$ and $\tilde I_2$ are intervals of length $\leq 3\delta$ that are $\sim \delta 2^k$ separated, $S_{I_1,I_2}$ is contained in the $\sim 2^{-k}$ neighborhood of a hyperplane that makes an angle $\leq 1/10$ with the $e_2,\ldots,e_n$ directions.  

Suppose that $x_3,x_3^\prime\in [0,1]$ satisfy
\begin{equation}\label{x3x3pContainedR}
(x_3,f(x_3))\in S_{I_1,I_2}\quad\textrm{and}\quad(x_3^\prime,f(x_3^\prime))\in S_{I_1,I_2}.
\end{equation}

Since $f$ is $1/10$-Lipschitz, we must have 
$$
|f(x_3)-f(x_3^\prime)|\leq \frac{1}{10}|x_3-x_3^\prime|.
$$
On the other hand, by \eqref{x3x3pContainedR} and the fact that $S_{I_1,I_2}$ is contained in the $\sim 2^{-k}$ neighborhood of a hyperplane that makes an angle $\leq 1/10$ with the $e_2,\ldots,e_n$ directions, we have
$$
|f(x_3)-f(x_3^\prime)|\geq 10|x_3-x_3^\prime|-O(2^{-k}).
$$

we conclude that $|x_3-x_3^\prime|\lesssim 2^{-k}$. This establishes \eqref{numberOfI3}. We conclude that \eqref{deltaCoveringZ} holds, so $Z$ has lower Minkowski dimension at most 2. 

By Theorem \ref{mainTheorem}, there is a set $X_1\subset[0,1]$ of Hausdorff dimension $1/2$ so that for each distinct $x_1,x_2,x_3\in X,$ we have $(x_1,x_2,x_3)\not\in Z$. This is precisely the statement that for each $x_1,x_2,x_3\in X$, the points $(x_1,f(x_1)),\ (x_2,f(x_2))$, and $(x_3,f(x_3))$ do not form the vertices of an isosceles triangle. To complete the proof, let $X = X/(10M)$.
\end{proof}

\bibliographystyle{amsplain}
\bibliography{FractalsAvoidingFractalSetsPaper}

\end{document}